\documentclass[reqno]{amsart}
\usepackage{amssymb,latexsym,amsmath,amsthm,enumerate,amsbsy}
\usepackage[mathscr]{eucal}
\usepackage{framed,color,graphicx}
\usepackage{mathrsfs}
\usepackage[all]{xy}
\usepackage{tikz}
\usetikzlibrary{positioning,decorations.pathreplacing,patterns,decorations.pathmorphing}
\tikzset{%
element/.style={draw, shape=circle, fill=white, inner sep=1.4pt}
}
\DeclareSymbolFont{bbold}{U}{bbold}{m}{n}
\DeclareSymbolFontAlphabet{\mathbbold}{bbold}

\theoremstyle{plain}
\newtheorem{thm}{Theorem}[section]
\newtheorem{lem}[thm]{Lemma}
\newtheorem{cor}[thm]{Corollary}
\newtheorem{pro}[thm]{Proposition}

\newtheorem{problem}[thm]{Problem}

\theoremstyle{definition}

\newtheorem{remark}[thm]{Remark}

\begin{document}

\title[The additively idempotent semiring $S_7^0$]
{The additively idempotent semiring $S_7^0$ is nonfinitely based}

\author{Yanan Wu}
\address{School of Mathematics, Northwest University, Xi'an, 710127, Shaanxi, P.R. China}
\email{wuyanan@stumail.nwu.edu.cn}
\author{Miaomiao Ren}
\address{School of Mathematics, Northwest University, Xi'an, 710127, Shaanxi, P.R. China}
\email{miaomiaoren@yeah.net}
\author{Xianzhong Zhao}
\address{School of Mathematics, Northwest University, Xi'an, 710127, Shaanxi, P.R. China}
\email{zhaoxz@nwu.edu.cn}

\thanks{Miaomiao Ren, corresponding author, is supported by National Natural Science Foundation of China (11701449)
and Natural Science Foundation of Shaanxi Province (2022JM-009).
Xianzhong Zhao is supported by National Natural Science Foundation of China (11971383).
}

\begin{abstract}
We show that the additively idempotent semiring $S_7^0$ has no finite basis for its equational theory.
This answers an open problem posed by Jackson et al. (J. Algebra 611 (2022), 211--245).
\end{abstract}

\maketitle

\section{Introduction and preliminaries}
An \emph{semiring} is an algebra $(S, +, \cdot)$ with two binary operations $+$ and $\cdot$
such that the additive reduct $(S, +)$ is a commutative semigroup,
the multiplicative reduct $(S, \cdot)$ is a semigroup and $(S, +, \cdot)$ satisfies the distributive laws
\[
(x+y)z\approx xy+xz, x(y+z)\approx xy+xz.
\]
An \emph{additively idempotent semiring} (ai-semiring for short) is a semiring such that its additive reduct is a semilattice,
that is, a commutative idenpotent semigroup. Such an algebra is also called a \emph{semilattice-ordered semigroup}.
This family of semirings includes many famous algebras:
the Kleene semiring of regular languages \cite{con}, tropical semirings \cite{pin98}, syntactic semirings of languages \cite{lp03}
and endomorphism semirings of semilattices \cite{dol09a}.
These and other similar algebras have broad applications in algebraic geometry, tropical geometry, idempotent analysis, information science and
theoretical computer science (see \cite{cc, gl, go, km, ms}).

A variety is \emph{finitely based} if it can be defined by a finite number of identities.
Otherwise, it is \emph{nonfinitely based}. An algebra $A$ is finitely based (resp. nonfinitely based)
if the variety $\mathsf{V}(A)$ generated by $A$ is finitely based (resp. nonfinitely based).
The finite basis problem is one of the most
important problems in the theory of varieties. Being very natural by itself, this problem also
has many interesting and unexpected connections with formal languages \cite{al95} and
classical number-theoretic conjectures \cite{per89}.

In the past two decades, the theory of ai-semiring varieties has been intensively studied and well developed.
Dolinka \cite{dol07, dol09} found the first example of a nonfinitely based ai-semiring variety and
gave a sufficient condition under which an ai-semiring variety is inherently nonfinitely based.
Pastijn et al. \cite{gpz05, pas05} classified the ai-semiring varieties satisfying $x^2\approx x$ and showed
that there are 78 such varieties.
Ren et al. \cite{rzw, rz16} proved that there are 179 ai-semiring varieties satisfying $x^3\approx x$.
Ren et al. \cite{rzs20} showed that if $n-1$ is square-free, then there are $2+2^{r+1}+3^r$ ai-semiring varieties
satisfying $x^n\approx x$ and $xy\approx yx$, where $r$ denotes the number of prime divisors of $n-1$.
Gusev and Volkov \cite{gv22a, gv22b} considered the finite basis problem for some ai-semirings,
which include the power semirings of finite groups and the
finite ai-semirings whose multiplicative reducts are inverse semigroups.
Recently, Ren et al. \cite{rjzl} provided an infinite series of minimal nonfinitely based ai-semiring varieties.

From \cite{gpz05, pas05, rzs20, rzw, rz16} one can observe that the following constructions
have played an important role in the theory of ai-semiring varieties.
Let $S$ be an ai-semiring. If we adjoin an extra element $0$ to the set $S$ and define
\[
(\forall a\in S\cup \{0\}) \quad a+0=a, a0=0a=0,
\]
then $S\cup \{0\}$ becomes an ai-semiring and is denoted by $S^0$.
In this semiring
we always have the implications
\begin{center}
$ab=0$ $\Rightarrow$ $a=0$ or $b=0$
\end{center}
and
\begin{center}
$a+b=0$ $\Rightarrow$ $a=b=0$.
\end{center}
It is easy to see that
$0$ is the identity element of $(S^0, +)$ and is the zero element of $(S^0, \cdot)$.

Recall that $S_7$ is a $3$-element ai-semiring $\{1, a, 0\}$ with the Cayley tables
\begin{center}
\begin{tabular}{c|ccc}
$+$&1&$a$&0\\
\hline
$1$&1&$0$&0\\
$a$&0&$a$&0\\
$0$&0&$0$&0\\
\end{tabular}\qquad
\begin{tabular}{c|ccc}
$\cdot$&1&$a$&0\\
\hline
$1$&1&$a$&0\\
$a$&$a$&$0$&0\\
$0$&0&$0$&0\\
\end{tabular}
.
\end{center}
To avoid ambiguity, we shall use $\infty$ to denote the multiplicative zero element of $S_7^0$.
So it has the following Cayley tables
\begin{center}
\begin{tabular}{c|cccc}
$+$&1&$a$&0&$\infty$\\
\hline
$1$&1&$0$&0&1\\
$a$&0&$a$&0&$a$\\
$0$&0&$0$&0&0\\
$\infty$&1&$a$&0&$\infty$\\
\end{tabular}\qquad
\begin{tabular}{c|cccc}
$\cdot$&1&$a$&0&$\infty$\\
\hline
$1$&1&$a$&0&$\infty$\\
$a$&$a$&$0$&0&$\infty$\\
$0$&0&$0$&0&$\infty$\\
$\infty$&$\infty$&$\infty$&$\infty$&$\infty$\\
\end{tabular}
.
\end{center}
Jackson et al. \cite{jrz} showed that $S_7$ has the following remarkable property.
Up to isomorphism, it is the only nonfinitely based ai-semiring of order at most three
Also, it can infect the nonfinitely based property to many other finite ai-semirngs.
This prompted Jackson et al. to propose the following interesting problems about $S_7^0$:
\begin{problem}\label{pro1}
\hspace*{\fill}
\begin{itemize}
\item[$(1)$] Is $S_7^0$ finitely based or nonfinitely based?

\item[$(2)$] What is the cardinality of the interval $[\mathsf{V}(S_7), \mathsf{V}(S_7^0)]$?
\end{itemize}
\end{problem}
In this paper we shall answer Problem \ref{pro1} (1) and show that $S_7^0$ is nonfinitely based.
For this purpose, the following notions and notations are necessary.
Let $X$ denote a countably infinite set of variables and $X^+$
the free semigroup on $X$. Then the ai-semiring
$(P_f(X^+), \cup, \cdot)$ consisting of all non-empty finite subsets of $X^+$
is free in the variety of all ai-semirings on $X$ (see \cite[Theorem 2.5]{kp}).
An \emph{ai-semiring identity} over $X$ is an
expression of the form
\[
u\approx v,
\]
where $u, v\in P_f(X^+)$.
For convenience, we always write
\[
u_1+\cdots+u_k\approx v_1+\cdots+v_\ell
\]
for the ai-semiring identity
\[
\{u_i \mid 1 \leq i \leq k\}\approx \{v_j \mid 1 \leq j \leq \ell\}.
\]
An \emph{ai-semiring substitution} is an endomorphism of $P_f(X^+)$.

Suppose that $\Sigma$ is a set of ai-semiring identities which include the identities that determine the variety of all ai-semirings.
Let $u\approx v$ be an ai-semiring identity such that
$u=u_1+\cdots+u_k$, $v=v_1+\cdots+v_\ell$, where $u_i$, $v_j\in X^+$, $1\leq i\leq k$, $1\leq j\leq \ell$.
Then it is easy to
see that the ai-semring variety defined by $u\approx v$ is equal to the ai-semring
variety defined by the simpler identities
$u\approx u+v_j, v\approx v+u_i, 1\leq i\leq k, 1\leq j\leq \ell$.
Therefore, to show that $u\approx v$ is derivable from $\Sigma$, we only need to show that
$u\approx u+v_j, v\approx v+u_i, 1\leq i\leq k, 1\leq j\leq \ell$ can be derived
from $\Sigma$. This is one of the most fundamental techniques in the theory
of ai-semiring varieties.

The following result about the equational logic of ai-semirings can be found in \cite[Lemma 3.1]{dol09}:
\begin{lem}\label{lem1}
Let $\Sigma$ be a set of ai-semiring identities and let $u\approx v$ be a nontrivial ai-semiring identity.
Then $u\approx v$ is derivable from $\Sigma$ if and only if
there exist $T_1, T_2,\ldots, T_n \in P_f(X^+)$
such that $u=T_1$, $v=T_n$ and, for every $i=1, 2,\ldots, n-1$, there are $A_i, B_i, P_i, Q_i, R_i\in P_f(X^+)$
and an ai-semiring substitution $\varphi_i: P_f(X^+) \to P_f(X^+)$ such that
\[
T_i=P_i\varphi_i(A_i)Q_i+R_i,~T_{i+1}=P_i\varphi_i(B_i)Q_i+R_i,
\]
where $A_i\approx B_i\in \Sigma$ or $B_i\approx A_i\in \Sigma$, $P_i$ and $Q_i$ may be the set $\{1\}$, $R_i$ may be the empty set.
\end{lem}

Let $S$ be a trivial ai-semiring, that is, $|S|=1$.
Then $S^0$ is a 2-element distributive lattice and is denoted by $D_2$.
Let $\omega$ be a word. Then $c(\omega)$ denotes the set of variables that occur in $\omega$.
The solution of the equational problem for $D_2$
can be found in \cite[Lemma 1.1 (iv)]{sr}:

\begin{lem}\label{nlemma1}
Let
$u\approx u+q$ be an ai-semiring identity such that
$u=u_1+\cdots+u_n$, where $u_i, q\in X^+$, $1\leq i \leq n$. Then
$u\approx u+q$ holds in $D_2$ if and only if $c(u_i)\subseteq c(q)$ for some $u_i \in u$.
\end{lem}

The following result demonstrates the importance of $D_2$ in studying the variety of the form $\mathsf{V}(S^0)$.
\begin{pro}\label{nnnpro}
If $S$ is an ai-semiring, then $D_2$ is a member of $\mathsf{V}(S^0)$.
\end{pro}
\begin{proof}
Consider the equivalence relation $\rho$ on $S^0$ given by
\begin{center}
$(\forall a, b\in S^0) \quad (a, b)\in \rho\Leftrightarrow$
either $a, b\in S$ or $a=b=0$.
\end{center}
It is easily verified that $\rho$ is a semiring
congruence on $S^0$ and that the quotient semiring
$S^0/ \rho$ is isomorphic to $D_2$.
Thus $D_2$ is a member of $\mathsf{V}(S^0)$ as required.
\end{proof}

Let $u=u_1+\cdots+u_n$ be an ai-semiring term, where $u_i \in X^+$, $1\leq i \leq n$.
If $q \in X^+$,
then $D_q(u)$ denotes the set of all $u_i$ in $u$ such that $c(u_i) \subseteq c(q)$.
The following result explores the relationship between the equational theories of $S^0$ and $S$.
\begin{pro}\label{lemma11}
Let $S$ be an ai-semiring and $u\approx u+q$ an ai-semiring identity such that
$u=u_1+\cdots+u_n$, where $u_i, q\in X^+$, $1\leq i \leq n$.
Then $u\approx u+q$ holds in $S^0$ if and only if
$D_q(u)\neq \emptyset$ and
$D_q(u)\approx D_q(u)+q$
is satisfied by $S$.
\end{pro}
\begin{proof}
Assume that $u\approx u+q$ holds in $S^0$.
By Proposition \ref{nnnpro} we have that $\mathsf{V}(S^0)$ contains $D_2$.
It follows that $D_2$ also satisfies $u\approx u+q$ and so by Lemma \ref{nlemma1}
$c(u_i)\subseteq c(q)$ for some $u_i\in u$.
Thus $u_i \in D_{q}(u)$ and so $D_{q}(u)\neq \emptyset$.
Let $\varphi: X \to S$ be an arbitrary substitution.
Consider $\psi: X \to S^0$ given by $\psi(x)=0$ for all $x \notin c(q)$
and $\psi(x)=\varphi(x)$ for all $x \in c(q)$.
Then $\psi(u)=\psi(u+q)$, since $u\approx u+q$ holds in $S^0$.
If we write $u=u'+D_{q}(u)$ such that $u'\cap D_{q}(u)=\emptyset$,
then for every $u_j$ in $u'$, $c(u_j)\nsubseteq c(q)$.
This implies that $\psi(u_j)=0$ and so is $\psi(u')$.
We now have
\begin{align*}
\varphi(D_{q}(u))
              & =0+\psi(D_{q}(u))\\
              & =\psi(u')+\psi(D_{q}(u))\\
              & =\psi(u'+D_{q}(u))\\
              & =\psi(u)=\psi(u+q)\\
              & =\psi(u'+D_{q}(u)+q)\\
              & =\psi(u')+\psi(D_{q}(u)+q)\\
              & =0+\varphi(D_{q}(u)+q)\\
              & =\varphi(D_{q}(u)+q).
\end{align*}
Thus $D_{q}(u)\approx D_{q}(u)+q$
is satisfied by $S$.

Conversely, let $\varphi: X \to S^0$ be an arbitrary substitution
and $Z$ denote the set $\{x \mid x\in c(u)\cup c(q), \varphi(x)=0\}$.
If $Z \cap c(q)\neq \emptyset$, then $\varphi(q)=0$ and so
\[\varphi(u+q)=\varphi(u)+\varphi(q)=\varphi(u)+0=\varphi(u).\]
Otherwise, $Z \cap c(q)= \emptyset$. Then
$Z\subseteq c(u)\backslash c(q)$ and so
$D_{q}(u)\subseteq D_Z(u)$, where $D_Z(u)$ denotes the set of all $u_i$ in $u$ such that
$c(u_i) \cap Z=\emptyset$.
Notice that both $\varphi(D_Z(u))$ and $\varphi(D_Z(u)+q)$ are elements of $S$.
By assumption it follows that $\varphi(D_Z(u))=\varphi(D_Z(u)+q)$.
We now have
\begin{align*}
\varphi(u)
              & =\varphi(u+D_Z(u))\\
              & =\varphi(u)+\varphi(D_Z(u))\\
              & =\varphi(u)+\varphi(D_Z(u)+q)\\
              & =\varphi(u+D_Z(u)+q)\\
              & =\varphi(u+q).
\end{align*}
Thus $u\approx u+q$ holds in $S^0$.
\end{proof}

\begin{remark}
The similar results with Proposition \ref{lemma11} have been appeared in \cite[Proposition 69]{aei}
and \cite[Lemma 2.5]{gpz05}.
\end{remark}

\section{The semiring $S_7^0$}

In this section we shall show that the ai-semiring $S_7^0$ is nonfinitely based in the syntactic way.

Let $\omega$ be a word and $x$ a letter. Then
\begin{itemize}
\item $\ell(\omega)$ denotes the number of variables occurring in $\omega$ counting multiplicities.

\item $occ(x, \omega)$ denotes the number of occurrences of $x$ in $\omega$.
\end{itemize}
So $\ell(\omega)$ is equal to the sum of $occ(x, \omega)$, $x\in c(\omega)$.
$\omega$ is \emph{linear} if $occ(x, \omega)=1$ for all $x \in c(\omega)$.
Let $u$ be an ai-semiring term. Then
\begin{itemize}
\item $c(u)$ denotes the set of variables that occur in $u$, that is, the union of $c(\omega)$, $\omega \in u$.

\item $\delta(u)$ denotes the set of nonempty subsets $Z$ of $c(u)$ such that for every $\omega \in u$,
$Z\cap c(\omega)$ is a singleton and $occ(x, \omega)=1$ if $\{x\}=Z\cap c(\omega)$.
\end{itemize}
The following result, which provides the solution of the equational problem for $S_7$,
is due to Jackson et al. \cite[Proposition 5.5]{jrz}:
\begin{lem}\label{lemma21}
Let $u\approx v$ be an ai-semiring identity. Then $u\approx v$ holds in $S_7$
if and only if $c(u)=c(v)$ and $\delta(u)=\delta(v)$.
\end{lem}

Next, we shall give the the solution of the equational problem for $S_7^0$.
\begin{pro}\label{nnpro2}
Let $u\approx u+q$ be an-semiring identity such that $u=u_1+\cdots+u_n$, where $q, u_i\in X^+$, $1\leq i\leq n$.
Then $u\approx u+q$ holds in $S_7^0$ if and only if $D_q(u)\neq \emptyset$,
$c(D_q(u))=c(q)$ and $\delta(D_q(u))=\delta(D_q(u)+q)$.
In particular, if $c(u)=c(q)$ and $\delta(u)=\emptyset$,
then $u\approx u+q$ is satisfied by $S^0_7$.
\end{pro}
\begin{proof}
This follows from Proposition \ref{lemma11} and Lemma \ref{lemma21} immediately.
\end{proof}

Let $\mathcal{V}$ be a variety. Then $\mathsf{Id}(\mathcal{V})$ denotes the set of all identities
over $X$ that hold in $\mathcal{V}$.
We now establish the main result of this section:
\begin{thm}\label{thm1}
Let $\mathcal{V}$ be an ai-semiring variety. If $\mathcal{V}$ is a subvariety of $\mathsf{V}(S_7^0)$ and
contains $S_7$ and $D_2$, then $\mathcal{V}$ is nonfinitely based.
\end{thm}

\begin{proof}
The strategy of the proof is as follows.
We shall show that for every $n \geq 1$, the set $\Sigma_n$
of all $n$-variable identities in $\mathsf{Id}(\mathcal{V})$ does not form a basis for $\mathsf{Id}(\mathcal{V})$.
For this it is enough to prove that for every $n \geq 1$, there exists
an identity that is true in $\mathcal{V}$ such that it can not be derived by $\Sigma_n$.
Notice that $S^0_7$ satisfies the identity $xy\approx yx$.
We may assume that every ai-semiring term in the sequel is a finite nonempty set of words in a free commutative semigroup over $X$.

Let $n$ be an arbitrary positive integer,
and let $u^{(n)} \approx u^{(n)}+q^{(n)}$ be an ai-semiring identity such that
\[
u^{(n)}=x_1x_2+x_2x_3+\cdots+x_{2n}x_{2n+1}+x_{2n+1}x_1
\]
and
\[q^{(n)}=x_1x_2\cdots x_{2n+1}.
\]
It is easy to verify that $c(u^{(n)})=c(q^{(n)})$ and $\delta(u^{(n)})=\emptyset$.
By Proposition \ref{nnpro2} we have that the identity
$u^{(n)}\approx u^{(n)}+ q^{(n)}$ is satisfied by $S_7^0$ and so does $\mathcal{V}$.
Suppose by way of contradiction that $u^{(n)}\approx u^{(n)}+ q^{(n)}$ can be derived by $\Sigma_n$.
Then by Lemma \ref{lem1} there exists an ai-semiring identity $A \approx B$ in $\Sigma_n$
and an ai-semiring substitution $\varphi$ such that
$\varphi(A)$ is a subterm of $u^{(n)}$ or $\varphi(B)$ is a subterm of $u^{(n)}$.
Without loss of generality, assume that $\varphi(A)$ is a subterm of $u^{(n)}$.
Then $P\varphi(A)Q+R = u^{(n)}$ for some ai-semiring terms $P$, $Q$ and $R$,
where $P_i$ and $Q_i$ may be the set $\{1\}$, $R_i$ may be the empty set.
It is easy to see that $A$ satisfies the following conditions:
\begin{itemize}
\item[$(a)$] $\ell(\omega)\leq 2$ for all $\omega \in A$;

\item[$(b)$] $\omega$ is linear for all $\omega \in A$;

\item[$(c)$] $\omega_1 \leq \omega_2$ implies that $\omega_1=\omega_2$ for all $\omega_1, \omega_2 \in A$;

\item[$(d)$] $A$ does not contain any ai-semiring term of the form $x_1x_2+x_2x_3+\cdots+x_{2m}x_{2m+1}+x_{2m+1}x_1$ for any $m\geq 1$.
\end{itemize}
Let $A'$ denote the sum of all words in $A$ with length $2$.
If $A' = \emptyset$, then by Lemma \ref{lemma21} $A \approx B$ is trivial, a contradiction.
If $A' \neq \emptyset$,
then it may be thought of as a graph
whose vertex set is $c(A')$ and edge set consists of $\{x, y\}$ if $xy$ lies in $A'$.
Recall that a graph is \emph{bipartite} if its vertex set can be decomposed into two disjoint sets such
that no two vertices within the same set are adjacent.
By \cite[Theorem 4]{bol98} we know that a graph is bipartite if and only if it does not contains an odd cycle.
From these observations one can deduce that $\delta(A)\neq \emptyset$
and that every variable in $c(A)$ lies in some set in $\delta(A)$.

Let $p$ be an arbitrary word in $B$. Then the identity $A \approx A+p$ is satisfied by $\mathcal{V}$
and so does hold in $S_7$ and $D_2$. It follows from Lemmas \ref{nlemma1} and \ref{lemma21}
that $c(p) \subseteq c(A)$, $\delta(A) = \delta(A+p)$ and
$c(\omega)\subseteq c(p)$ for some $\omega \in A$.
This implies that $p$ contains $\omega$.
Assume that $\omega$ is a proper subword of $p$. Then $p=\omega\omega'$ for some $\omega' \in X^+$.
Furthermore, every variable in $\omega'$ does not not lie in any set in $\delta(A)$,
a contradiction. We conclude that $p=\omega$.
Thus $B \subseteq A$ and so $\varphi(B) \subseteq \varphi(A)$.
Since $P\varphi(A)Q+R = u^{(n)}$,
we have that $P\varphi(B)Q+R \subseteq u^{(n)}$, a contradiction.
Hence by Lemma \ref{lem1} $u^{(n)}\approx u^{(n)}+ q^{(n)}$ can not be derived by $\Sigma_n$
and so $\mathcal{V}$ is nonfinitely based.
\end{proof}

By Theorem \ref{thm1} we immediately have the following two corollaries:
\begin{cor}
The ai-semiring $S_7^0$ is nonfinitely based.
\end{cor}

\begin{cor}
The join of $\mathsf{V}(S_7)$ and $\mathsf{V}(D_2)$ is nonfinitely based.
\end{cor}

\section{Conclusion}
We have shown that $S_7^0$ is nonfinitely based. This answers an open problem
proposed by Jackson et al. \cite{rjzl}. By contrast, we know very little information
about the cardinality of the interval $[\mathsf{V}(S_7), \mathsf{V}(S_7^0)]$.
One can prove that it has a lower bound $3$. More precisely,
this interval contains the join of $\mathsf{V}(S_7)$ and $\mathsf{V}(D_2)$,
which is distinct from $\mathsf{V}(S_7)$ and $\mathsf{V}(S_7^0)$.
This follows from the fact that the identity
\[
x^2+y \approx x^2y^2
\]
is satisfied by
$S_7$, but does not hold in $D_2$; the identity
\[
x^2+y \approx x^2+y+y^2
\]
is satisfied by
$S_7$ and $D_2$, but does not hold in $S_7^0$.
Let $\mathcal{V}$ be a variety that properly contains $\mathsf{V}(S_7)$ and is a subvariety of $\mathsf{V}(S_7^0)$.
We do not know whether $\mathcal{V}$ must contain $D_2$.
If the answer is positive, then every variety in the interval $[\mathsf{V}(S_7), \mathsf{V}(S_7^0)]$
is nonfinitely based, although the cardinality of this interval is not known.

\bibliographystyle{amsplain}

%%%%%%%%%%%%%%%%%%%%%%%%%%%%%%%%%%%%%%%%%%%%%%%%%%%%%%%%%%%%%%%%%%
%%%%%%%%%%%%%%%%%%%%%%%%%%%%%%%%%%%%%%%%%%%%%%%%%%%%%%%%%%%%%%%%%%

\end{document}